\documentclass[11pt,a4paper]{article}

\usepackage[a4paper,margin=29mm]{geometry}
\usepackage[T1]{fontenc}
\usepackage[utf8]{inputenc}
\usepackage{lmodern}
\usepackage{microtype}
\usepackage{amsmath,amssymb,amsthm,mathtools}
\usepackage{enumitem}
\usepackage{hyperref}

\hypersetup{
  colorlinks=true,
  linkcolor=black,
  citecolor=black,
  urlcolor=black,
  pdftitle={Exact output statistics of Icart's encoding in the exceptional \(j=0\) case},
  pdfauthor={David Kumallagov}
}

\theoremstyle{plain}
\newtheorem{theorem}{Theorem}[section]
\newtheorem{proposition}[theorem]{Proposition}

\newtheorem{corollary}[theorem]{Corollary}
\theoremstyle{remark}
\newtheorem{remark}[theorem]{Remark}

\DeclareMathOperator{\Col}{Col}

\newcommand{\Fq}{\mathbb F_q}
\newcommand{\Fqs}{\mathbb F_q^*}
\newcommand{\Ocal}{\mathcal O}
\newcommand{\TV}{\operatorname{TV}}
\newcommand{\Prob}{\mathbb P}
\newcommand{\Exp}{\mathbb E}

\title{Exact output statistics of Icart's encoding in the exceptional \(j=0\) case}
\author{David Kumallagov\\
\small Information Technologies and Programming Faculty, ITMO University\\ \small Saint Petersburg, Russia\\ \small \texttt{kumdavid95@gmail.com}}
\date{}

\begin{document}
\maketitle

\begin{abstract}
Icart's encoding is a classical deterministic map from finite fields to
elliptic curves and a basic ingredient in early hash-to-curve constructions.
We determine the exact one-output distribution of this map in the exceptional
\(j=0\) case.  More precisely, for
\[
  E_{0,b}:Y^2=X^3+b,\ q\equiv2\pmod3,
\]
we compute the complete fibre distribution of
\(f_{0,b}:\mathbb F_q\to E_{0,b}(\mathbb F_q)\).  This gives closed formulae
for the image size, total variation distance from uniform, collision
probability, power sums, entropy measures and basic batch statistics.  We also
derive the exact second moment of all nontrivial character sums of the output
distribution.  Via the Weil pairing, this becomes an exact energy formula for
pairing-character tests on the supersingular \(j=0\) family whose odd prime order subgroups have
embedding degree two.
\end{abstract}



\section{Introduction}

Deterministic maps from finite fields to elliptic curves are basic components
in the arithmetic of hash-to-curve constructions. In the terminology codified
in RFC~9380, a map-to-curve algorithm sends field elements to curve points,
whereas uniform hash-to-curve constructions typically combine several such
outputs and apply the necessary subgroup operations \cite{RFC9380}. The
one-output distribution of a deterministic map is governed by its fibres:
the probability of a point is exactly the number of its preimages divided by
the size of the field. Exact fibre counts therefore give the most direct
description of the nonuniform distribution induced by a single map output.

Icart's map \cite{Icart} is a classical deterministic encoding for short
Weierstrass curves over fields with \(q\equiv2\pmod3\), where the cube map is
bijective. Icart-type maps subsequently played a central role in the
development of indifferentiable hashing to elliptic curves, from the original
construction based on such encodings to the later character-sum framework for
deterministic encodings \cite{BCIMRT,FFSTV}. The image size of Icart's map was
studied in \cite{FSV}, where the generic case \(a\ne0\) was shown to have
density \(5/8\), while the exceptional case \(a=0\) has density \(2/3\), both
up to \(O(\sqrt q)\) error terms.

This paper computes the exceptional case exactly. Let
\[
  E_{0,b}:Y^2=X^3+b,\qquad b\ne0,\qquad q\equiv2\pmod3,
\]
and let \(f_{0,b}:\mathbb F_q\to E_{0,b}(\mathbb F_q)\) be Icart's original
map. We determine the number of curve points with \(0,1,2\) and \(3\)
preimages under \(f_{0,b}\), taking into account the special convention
\(f_{0,b}(0)=\mathcal O\). This yields the exact fibre enumerator of the map.

Let \(\chi\) be the quadratic character of \(\mathbb F_q\), extended by
\(\chi(0)=0\), and put
\[
  \alpha=\chi(-6b),\qquad
  \beta=\chi(-2b),\qquad
  \varepsilon=\frac{1-\alpha}{2},\qquad
  \delta=\frac{1+\alpha}{2}.
\]
If \(M_j\) denotes the number of affine points of \(E_{0,b}(\mathbb F_q)\)
with exactly \(j\) preimages, then our main theorem gives
\[
  M_0=\frac{q+1}{3}+\varepsilon,\qquad
  M_1=\frac{q-2-\beta}{2}-\varepsilon,
\]
\[
  M_2=1+\beta+\delta,\qquad
  M_3=\frac{q-2-3\beta}{6}-\delta,
\]
and the point at infinity has one preimage. In particular,
\[
  \#\operatorname{Im}(f_{0,b})
  =
  \frac{2(q+1)}3-\frac{1-\chi(-6b)}2.
\]
Thus the \(O(\sqrt q)\) estimate in the exceptional case is replaced by an
exact formula.

The same fibre enumerator determines the main statistical invariants of one
Icart output. If \(\mu_b\) is the distribution of \(f_{0,b}(U)\) for a uniform
\(U\in\mathbb F_q\), and \(\nu_b\) is the uniform distribution on
\(E_{0,b}(\mathbb F_q)\), then
\[
  \|\mu_b-\nu_b\|_{\rm TV}
  =
  \frac{2q+5-3\chi(-6b)}{6(q+1)}.
\]
We also obtain closed formulae for the collision probability, all integral
power sums, entropy measures and basic batch statistics.

Finally, the exceptional family \(Y^2=X^3+b\), \(q\equiv2\pmod3\), is the standard supersingular \(j=0\) family whose odd prime order subgroups
have embedding degree two. This allows the
Fourier statistics of \(\mu_b\) to be interpreted through Weil-pairing
characters. We compute the exact second moment of all nontrivial character
sums; equivalently, we obtain an exact energy formula for the corresponding
embedding-degree-two pairing tests. Recent work on pairing computation for
embedding-degree-two curves provides additional cryptographic context for this
interpretation \cite{ZhengLinZhao2026}.

\section{Icart's map and the inverse cubic}
\label{sec:icart}

Let $q=p^n$ with $p\ge5$, and assume $q\equiv2\pmod3$.  Then the cube map on
$\Fq$ is a bijection.  Let
\[
  E_{a,b}:\quad Y^2=X^3+aX+b, \qquad 4a^3+27b^2\ne0,
\]
be an elliptic curve over $\Fq$.  For $u\ne0$, Icart's map is
\begin{equation}\label{eq:icart-map}
  v=\frac{3a-u^4}{6u}, \qquad
  x=\left(v^2-b-\frac{u^6}{27}\right)^{1/3}+\frac{u^2}{3},
  \qquad
  y=ux+v,
\end{equation}
and $f_{a,b}(0)=\Ocal$.  The formula $y=ux+v$ is the equation of the auxiliary
line used in the construction; the map itself is not linear.

The inverse-polynomial description is due to Icart and is used in the form
recorded by Farashahi--Shparlinski--Voloch \cite{Icart,FSV}.  For an affine
point $P=(x,y)\in E_{a,b}(\Fq)$ set
\begin{equation}\label{eq:quartic-H}
  H_{a,b}(P,U)=U^4-6xU^2+6yU-3a.
\end{equation}
If $a=0$, this quartic has the formal factor $U$, and after removing it one
obtains the cubic
\begin{equation}\label{eq:cubic-K}
  K_b(P,U)=U^3-6xU+6y.
\end{equation}
Then
\begin{equation}\label{eq:inverse-criterion}
  \#f_{0,b}^{-1}(P)=\#\{u\in\Fqs:K_b(P,u)=0\}
\end{equation}
for affine points $P\in E_{0,b}(\Fq)$, while $f_{0,b}^{-1}(\Ocal)=\{0\}$.  The
exclusion of $u=0$ in \eqref{eq:inverse-criterion} is essential: in Icart's
original map the input $u=0$ is sent to $\Ocal$, not to an affine point.

\section{Exact fibre counts for \texorpdfstring{$a=0$}{a=0}}
\label{sec:exact-a0}

From now on let
\[
  E_{0,b}:\quad Y^2=X^3+b,
  \qquad b\ne0.
\]
Since cubing is a bijection of $\Fq$, for each $y\in\Fq$ there is a unique
$x\in\Fq$ such that $x^3=y^2-b$.  Hence
\begin{equation}\label{eq:a0-cardinality}
  \#E_{0,b}(\Fq)_{\rm aff}=q,
  \qquad
  \#E_{0,b}(\Fq)=q+1.
\end{equation}

For an affine point $P=(x,y)$ write
\[
  K_P(U)=U^3-6xU+6y.
\]
Let $n_i$ be the number of affine points $P\in E_{0,b}(\Fq)$ for which $K_P$
has exactly $i$ distinct roots in $\Fq$, where the root $0$, if present, is
included.  The cubic $K_P$ never has a triple root because $b\ne0$.

\begin{proposition}\label{prop:n_i}
Let $\beta=\chi(-2b)$.  Then
\begin{align*}
  n_0&=\frac{q+1}{3}, &
  n_1&=\frac{q-2-\beta}{2},\\
  n_2&=1+\beta, &
  n_3&=\frac{q-2-3\beta}{6}.
\end{align*}
\end{proposition}

\begin{proof}
First,
\begin{equation}\label{eq:sum-ni}
  n_0+n_1+n_2+n_3=q.
\end{equation}
Count pairs $(P,u)$ with $P\in E_{0,b}(\Fq)_{\rm aff}$, $u\in\Fq$, and
$K_P(u)=0$.  On the one hand, the number of such pairs is
$n_1+2n_2+3n_3$.  On the other hand, for fixed $u\in\Fq$ put $z=x-u^2/3$.
The equations $K_P(u)=0$ and $y^2=x^3+b$ give
\[
  y=uz+\frac{u^3}{6},
  \qquad
  z^3=-b-\frac{u^6}{108}.
\]
For each $u$, this has a unique solution $z$, and hence a unique affine point
$P$.  Thus
\begin{equation}\label{eq:pair-count-1}
  n_1+2n_2+3n_3=q.
\end{equation}
Subtracting \eqref{eq:sum-ni} gives
\begin{equation}\label{eq:n0-n2-n3}
  n_0=n_2+2n_3.
\end{equation}

Next put $R(x)=x^3+9b$.  The discriminant of $K_P(U)=U^3-6xU+6y$ is
\begin{equation}\label{eq:disc-cubic}
  \operatorname{disc} K_P=-108\,R(x).
\end{equation}
Since $q\equiv2\pmod3$, the element $-3$ is not a square in $\Fq$, and hence
$-108$ is a nonsquare.  For a separable cubic over an odd finite field, the
discriminant is a square if and only if the Frobenius permutation of its roots
is even.  It follows that the points for which $K_P$ has exactly one root are
precisely those for which $R(x)$ is a nonzero square, while points with exactly
two distinct roots are precisely those with $R(x)=0$.

Count triples $(x,y,s)\in\Fq^3$ satisfying
\[
  y^2=x^3+b,
  \qquad
  s^2=x^3+9b.
\]
Subtracting gives $(s-y)(s+y)=8b$.  Since $b\ne0$, set $r=s-y\in\Fqs$.  Then
\[
  s=\frac{r+8b/r}{2},
  \qquad
  y=\frac{8b/r-r}{2},
\]
and after $y$ is known there is a unique $x=(y^2-b)^{1/3}$.  Hence the number
of triples is $q-1$.  In terms of the $n_i$, this number is
\begin{equation}\label{eq:pair-count-2}
  2n_1+n_2=q-1.
\end{equation}
Equations \eqref{eq:sum-ni}, \eqref{eq:pair-count-1} and
\eqref{eq:pair-count-2} imply $n_0=(q+1)/3$.

Finally, $n_2$ is the number of affine points for which $R(x)=0$.  The equation
$R(x)=0$ has one solution $x$, and for this $x$ the curve equation gives
$y^2=-8b$.  Therefore
\[
  n_2=1+\chi(-8b)=1+\chi(-2b)=1+\beta.
\]
The formulae for $n_1$ and $n_3$ follow from \eqref{eq:pair-count-2} and
\eqref{eq:n0-n2-n3}.
\end{proof}

We now pass from roots of $K_P$ to true fibres of Icart's map.  Let
\[
  r(P)=\#f_{0,b}^{-1}(P),\qquad P\in E_{0,b}(\Fq),
\]
and, for $j=0,1,2,3$, let
\[
  M_j=\#\{P\in E_{0,b}(\Fq)_{\rm aff}:r(P)=j\}.
\]
Keep the notation
\begin{equation}\label{eq:alpha-beta}
  \alpha=\chi(-6b),
  \qquad
  \beta=\chi(-2b),
  \qquad
  \varepsilon=\frac{1-\alpha}{2},
  \qquad
  \delta=\frac{1+\alpha}{2}.
\end{equation}

\begin{theorem}\label{thm:a0-fibres}
Let $q\equiv2\pmod3$, let $\operatorname{char}\Fq\ge5$, and let
$E_{0,b}:Y^2=X^3+b$ with $b\ne0$.  Then $r(\Ocal)=1$ and
\begin{align}
  M_0&=\frac{q+1}{3}+\varepsilon, \label{eq:M0}\\
  M_1&=\frac{q-2-\beta}{2}-\varepsilon, \label{eq:M1}\\
  M_2&=1+\beta+\delta, \label{eq:M2}\\
  M_3&=\frac{q-2-3\beta}{6}-\delta. \label{eq:M3}
\end{align}
\end{theorem}

\begin{proof}
By \eqref{eq:inverse-criterion}, the affine preimages of a point $P$ are the
nonzero roots of $K_P$.  The only affine point for which $0$ is a root is
\[
  P_0=(x_0,0),\qquad x_0^3=-b.
\]
For this point,
\[
  K_{P_0}(U)=U(U^2-6x_0),
\]
and
\[
  \chi(6x_0)=\chi(6)\chi(x_0)=\chi(6)\chi(x_0^3)=\chi(-6b)=\alpha.
\]
If $\alpha=-1$, then $K_{P_0}$ has the single root $0$, so $P_0$ moves from the
class counted by $n_1$ to the true fibre class $M_0$.  If $\alpha=1$, then
$K_{P_0}$ has the three roots $0,\pm\sqrt{6x_0}$, so $P_0$ moves from the class
counted by $n_3$ to the true fibre class $M_2$.  All other affine points have
the same true fibre size as the number of roots of $K_P$.  Hence
\[
  M_0=n_0+\varepsilon,
  \quad
  M_1=n_1-\varepsilon,
  \quad
  M_2=n_2+\delta,
  \quad
  M_3=n_3-\delta,
\]
and Proposition~\ref{prop:n_i} gives the displayed formulae.
\end{proof}

We write
$\mathcal I_{0,b}=f_{0,b}(\mathbb F_q)
  \subseteq E_{0,b}(\mathbb F_q)$
for the image of Icart's map.

\begin{corollary}\label{cor:a0-image}
Under the hypotheses of Theorem~\ref{thm:a0-fibres},
\begin{equation}\label{eq:a0-image-exact}
  \#\mathcal I_{0,b}
  =\frac{2(q+1)}3-\frac{1-\chi(-6b)}2.
\end{equation}
Consequently,
\begin{equation}\label{eq:a0-error-one}
  \left|\#\mathcal I_{0,b}-\frac23\#E_{0,b}(\Fq)\right|\le1.
\end{equation}
Moreover,
\begin{equation}\label{eq:five-eighths-false}
  \left|\#\mathcal I_{0,b}-\frac58\#E_{0,b}(\Fq)\right|
  =\left|\frac{q+1}{24}-\frac{1-\chi(-6b)}2\right|,
\end{equation}
so the coefficient $5/8$ cannot hold uniformly in the case $a=0$ with an
$O(\sqrt q)$ error term.
\end{corollary}

\begin{proof}
The affine points outside the image are exactly the $M_0$ points with no
preimage.  Therefore
\[
  \#\mathcal I_{0,b}=1+q-M_0=\frac{2(q+1)}3-\varepsilon.
\]
The remaining assertions follow from \eqref{eq:a0-cardinality}.
\end{proof}

\section{Output statistics}
\label{sec:statistics}

Let $G=E_{0,b}(\Fq).$
Let $U$ be uniformly distributed in $\Fq$, and write
$\mu_b(P)=\Prob[f_{0,b}(U)=P],\ P\in G.$
Let $\nu_b$ be the uniform distribution on $G$.

For two probability distributions $\rho$ and $\sigma$ on $G$, we use the
total variation distance
\[
  \|\rho-\sigma\|_{\TV}
  =
  \frac12\sum_{P\in G}|\rho(P)-\sigma(P)|.
\]
Here the subscript $\TV$ stands for total variation. 

For a $G$-valued random variable $X$, we write $\mathcal L(X)$ for its law:
\[
  \mathcal L(X)(P)=\Prob[X=P],
  \quad
  P\in G.
\]
For probability distributions $\rho$ and $\sigma$ on $G$, their convolution is
defined by
\[
  (\rho*\sigma)(P)
  =
  \sum_{Q\in G}\rho(Q)\sigma(P-Q),
  \qquad
  P\in G.
\]
Finally, we write
$\Col(\rho)=\sum_{P\in G}\rho(P)^2.$

Put
\[
  A_0=M_0,
  \qquad
  A_1=M_1+1,
  \qquad
  A_2=M_2,
  \qquad
  A_3=M_3.
\]
Thus $A_j$ is the number of curve points, including $\Ocal$, with exactly $j$
preimages under $f_{0,b}$.

\begin{corollary}\label{cor:fibre-TV}
The fibre enumerator
\[
  W_b(T)=\sum_{P\in E_{0,b}(\Fq)}T^{r(P)}
\]
is
\begin{equation}\label{eq:fibre-enumerator}
  W_b(T)=A_0+A_1T+A_2T^2+A_3T^3,
\end{equation}
where
\begin{align*}
  A_0&=\frac{q+1}{3}+\varepsilon, &
  A_1&=1+\frac{q-2-\beta}{2}-\varepsilon,\\
  A_2&=1+\beta+\delta, &
  A_3&=\frac{q-2-3\beta}{6}-\delta.
\end{align*}
Moreover,
\begin{equation}\label{eq:a0-TV}
  \|\mu_b-\nu_b\|_{\TV}
  =\frac{2q+5-3\chi(-6b)}{6(q+1)}.
\end{equation}
Equivalently, the distance is $1/3$ if $\chi(-6b)=1$, and it is
$1/3+1/(q+1)$ if $\chi(-6b)=-1$. 
\end{corollary}

\begin{proof}
The enumerator follows directly from Theorem~\ref{thm:a0-fibres} and the fact
that $\Ocal$ has one preimage.  Every point in the image has probability at
least $1/q>1/(q+1)$, while every point outside the image has probability $0$.
Therefore
\[
  \|\mu_b-\nu_b\|_{\TV}
  =\frac{\#(E_{0,b}(\Fq)\setminus\mathcal I_{0,b})}{q+1}
  =\frac{A_0}{q+1}.
\]
Substituting $A_0=(q+1)/3+(1-\chi(-6b))/2$ gives \eqref{eq:a0-TV}.
\end{proof}

\begin{corollary}\label{cor:powers-renyi}
For every integer $k\ge1$,
\begin{equation}\label{eq:all-power-sums}
  \sum_{P\in E_{0,b}(\Fq)}\mu_b(P)^k
  =\frac{A_1+2^kA_2+3^kA_3}{q^k}.
\end{equation}
In particular,
\begin{align}
  \Col(\mu_b)=\sum_P\mu_b(P)^2
  &=\frac{2q-2-\chi(-2b)-2\chi(-6b)}{q^2},\label{eq:a0-collision}\\
  \sum_P\mu_b(P)^3
  &=\frac{5q-11-6\chi(-2b)-9\chi(-6b)}{q^3}.
     \label{eq:a0-third-moment}
\end{align}
The Renyi entropy of order $k>1$ is
\begin{equation}\label{eq:renyi}
  H_k(\mu_b)=\frac{1}{1-k}
  \log\left(\frac{A_1+2^kA_2+3^kA_3}{q^k}\right).
\end{equation}
The Shannon entropy and min-entropy are
\begin{equation}\label{eq:shannon-entropy}
  H(\mu_b)=\log q-\frac{2A_2\log 2+3A_3\log 3}{q},
\end{equation}
and
\begin{equation}\label{eq:min-entropy}
  H_\infty(\mu_b)=\log q-
  \log\bigl(\max\{r\in\{1,2,3\}:A_r>0\}\bigr).
\end{equation}
Finally,
\begin{equation}\label{eq:chi-square}
  \chi^2(\mu_b\|\nu_b)
  =(q+1)\Col(\mu_b)-1
  =\frac{q^2-2-(\beta+2\alpha)(q+1)}{q^2}.
\end{equation}

\end{corollary}

\begin{proof}
A point with $r$ preimages has mass $r/q$.  Summing over the numbers $A_r$
gives \eqref{eq:all-power-sums}.  The collision and third-moment formulae
are the cases $k=2$ and $k=3$.  The entropy formulae follow from the standard
definitions.  The chi-square identity follows from
$\chi^2(\mu_b\|\nu_b)=(q+1)\sum_P\mu_b(P)^2-1$.
\end{proof}

\begin{corollary}\label{cor:batch}
Let $U_1,\ldots,U_m$ be independent uniform elements of $\Fq$.  If
\[
  C_m=\#\{1\le i<j\le m:f_{0,b}(U_i)=f_{0,b}(U_j)\},
\]
then
\begin{equation}\label{eq:expected-pair-collisions}
  \Exp C_m=\binom m2\Col(\mu_b).
\end{equation}
If
\[
  D_m=\#\{f_{0,b}(U_1),\ldots,f_{0,b}(U_m)\},
\]
then
\begin{equation}\label{eq:expected-distinct}
\Exp D_m=
\sum_{r=1}^3 A_r\left(1-\left(1-\frac rq\right)^m\right).
\end{equation}
\end{corollary}

\section{Pairing characters and Fourier energy}
\label{sec:pairing-characters}

The preceding formulae determine not only point probabilities, but also the
second moment of all character sums associated with the output distribution.
For elliptic curves, these characters may be realized through pairings.  In the
present exceptional family this interpretation has a concrete pairing-friendly
meaning: the curve is supersingular and its odd prime-order subgroups have
embedding degree two.

Let
\[
  G=E_{0,b}(\Fq),\qquad N=\#G=q+1.
\]

\begin{proposition}\label{prop:k2-pairing-tests}
Under the hypotheses of Theorem~\ref{thm:a0-fibres}, put
\(N=\#E_{0,b}(\Fq)=q+1\).  Then \(E_{0,b}\) is supersingular and
\(N\mid q^2-1\).  Hence \(\mu_N\subset \mathbb F_{q^2}^*\).  Moreover, for every
odd prime divisor \(r\mid N\), the embedding degree of \(E_{0,b}\) with respect
to \(r\) is exactly \(2\).

The Weil pairing
\[
  e_N:E_{0,b}[N]\times E_{0,b}[N]\longrightarrow \mu_N
\]
is perfect.  After fixing an embedding of \(\mu_N\) in the complex roots of
unity, the rule
\[
  Q\longmapsto \psi_Q,
  \qquad
  \psi_Q(P)=e_N(P,Q),\quad P\in G,
\]
induces an isomorphism
\[
  E_{0,b}[N]/G^\perp \simeq \widehat G,
  \qquad
  G^\perp=\{Q\in E_{0,b}[N]:e_N(P,Q)=1\text{ for all }P\in G\}.
\]
Thus all Fourier characters of \(G\) are realized by Weil-pairing tests; on the
odd prime-order components these are embedding degree-two pairing tests.
\end{proposition}

\begin{proof}
Equation~\eqref{eq:a0-cardinality} gives \(N=\#E_{0,b}(\Fq)=q+1\).  Therefore
every point of \(G\) is annihilated by \(N\), and \(N\mid q^2-1\).  Since
\(\mathbb F_{q^2}^*\) is cyclic of order \(q^2-1\), it contains \(\mu_N\).  If
\(r>2\) is a prime divisor of \(N\), then \(q\equiv -1\pmod r\), and hence the
multiplicative order of \(q\) modulo \(r\) is exactly \(2\).

The supersingularity assertion is the standard criterion for the \(j=0\) family \cite[Chapter V]{SilvermanAEC}
in characteristic at least \(5\): curves of the form \(Y^2=X^3+b\) are
supersingular when the characteristic is $\equiv 2 \pmod 3.$  Here
\(q=p^n\equiv2\pmod3\) implies \(p\equiv2\pmod3\) and \(n\) is odd.  Finally,
the Weil pairing is perfect on \(E_{0,b}[N]\) because \(p\nmid N\)
\cite[Chapter III, Section 8]{SilvermanAEC}.  The displayed quotient isomorphism
is the usual annihilator statement for a perfect pairing.
\end{proof}

This pairing interpretation identifies the Fourier energy below with an exact
second moment over concrete Weil pairing character tests.

For a character \(\psi\in\widehat G\), put
\[
  B_\psi=\sum_{u\in\Fq}\psi(f_{0,b}(u)).
\]
The trivial character gives \(B_1=q\).

\begin{theorem}\label{thm:pairing-energy}
Under the hypotheses of Theorem~\ref{thm:a0-fibres},
\begin{equation}\label{eq:pairing-energy}
  \sum_{\psi\ne1}|B_\psi|^2
  =q^2-2-\bigl(\chi(-2b)+2\chi(-6b)\bigr)(q+1).
\end{equation}
Equivalently, choosing one representative $Q$ in each nonzero coset of
$E_{0,b}[N]/G^\perp$, one has the Weil-pairing form
\begin{equation}\label{eq:weil-energy}
  \sum_{\overline Q\ne0}
  \left|\sum_{u\in\Fq} e_N(f_{0,b}(u),Q)\right|^2
  =q^2-2-\bigl(\chi(-2b)+2\chi(-6b)\bigr)(q+1).
\end{equation}
\end{theorem}

\begin{proof}
Let
\[
  \widehat\mu_b(\psi)=\sum_{P\in G}\mu_b(P)\psi(P)=\frac{B_\psi}{q}.
\]
Plancherel's identity for the finite abelian group $G$ gives
\[
  \sum_{\psi\in\widehat G}|\widehat\mu_b(\psi)|^2
  =\#G\sum_{P\in G}\mu_b(P)^2=(q+1)\Col(\mu_b).
\]
Since $\widehat\mu_b(1)=1$,
\[
  \sum_{\psi\ne1}|B_\psi|^2
  =q^2\bigl((q+1)\Col(\mu_b)-1\bigr).
\]
Substitution of \eqref{eq:a0-collision} gives \eqref{eq:pairing-energy}; the
Weil-pairing form follows from the realization of characters above.
\end{proof}

\begin{corollary}\label{cor:forced-sqrt-scale}
With $\alpha=\chi(-6b)$ and $\beta=\chi(-2b)$,
\begin{equation}\label{eq:average-energy}
  \frac1q\sum_{\psi\ne1}|B_\psi|^2
  =q-(\beta+2\alpha)-\frac{2+\beta+2\alpha}{q}.
\end{equation}
In particular, some nontrivial character satisfies
\[
  |B_\psi|\ge
  \left(q-(\beta+2\alpha)-\frac{2+\beta+2\alpha}{q}\right)^{1/2}.
\]
Thus the $\sqrt q$ scale in character-sum estimates for Icart's map is forced
on average by the exact one-output distribution.
\end{corollary}

\section{Two-output smoothing}
\label{sec:two-output}

The exact formula \eqref{eq:a0-TV} shows that one Icart output is far from
uniform.  The standard character-sum argument explains why adding two
independent outputs smooths the distribution.  The character-sum framework of \cite{FFSTV} gives, for Icart's map,
\begin{equation}\label{eq:FFSTV-bound}
  |S_f(\psi)|\le B_q:=12\sqrt q+3,
  \qquad
  S_f(\psi)=\sum_{u\in\Fq}\psi(f_{0,b}(u)),
\end{equation}
for every nontrivial character $\psi$ of $E_{0,b}(\Fq)$ \cite{FFSTV}.
Combining this estimate with Theorem~\ref{thm:pairing-energy} gives the
following explicit form of the usual two-output smoothing bound.

\begin{proposition}\label{prop:two-output}
Let $U_1,U_2$ be independent uniform elements of $\Fq$, and let
\[
  Z=f_{0,b}(U_1)+f_{0,b}(U_2).
\]
Then
\begin{equation}\label{eq:two-output-bound}
  \|\mathcal L(Z)-\nu_b\|_{\TV}
  \le
  \frac{12\sqrt q+3}{2q^2}
  \sqrt{q^2-2-(\beta+2\alpha)(q+1)}.
\end{equation}
In particular,
\[
  \|\mathcal L(Z)-\nu_b\|_{\TV}=O(q^{-1/2}).
\]
\end{proposition}

\begin{proof}
For $\widehat\mu_b(\psi)=S_f(\psi)/q$, the Fourier transform of
$\mu_b*\mu_b$ is $\widehat\mu_b(\psi)^2$.  By Plancherel and Cauchy's
inequality on $G$,
\[
  \|\mu_b*\mu_b-\nu_b\|_{\TV}
  \le \frac12
  \left(\sum_{\psi\ne1}|\widehat\mu_b(\psi)|^4\right)^{1/2}.
\]
Using \eqref{eq:FFSTV-bound} for one factor and Theorem~\ref{thm:pairing-energy}
for the other gives \eqref{eq:two-output-bound}.
\end{proof}

\begin{remark}
Proposition~\ref{prop:two-output} is not a new indifferentiability theorem.  It
is the standard character-sum smoothing argument specialized with the exact
one-output energy computed here.
\end{remark}

\section{Related work and scope}
\label{sec:related-work}

The construction of deterministic points on elliptic curves over finite fields
was studied by Skalba \cite{Skalba05}, Shallue--van de Woestijne \cite{SW06},
Icart \cite{Icart}, and in hyperelliptic settings by
Kammerer--Lercier--Renault \cite{KLR10}.  Image sizes and distributional
properties of such maps were studied by Farashahi--Shparlinski--Voloch
\cite{FSV}, Fouque--Tibouchi \cite{FT10}, and in special models such as
Hessian curves by Farashahi \cite{FarashahiHessian}.

The indifferentiable-hashing line begins, for our purposes, with the Icart-based
construction of Brier et al. \cite{BCIMRT} and the character-sum framework of \cite{FFSTV}.  Practical
standardized hash-to-curve suites use SWU-type maps, isogenies, domain
separation and cofactor clearing, as in RFC 9380 \cite{RFC9380}; see also
Fouque--Tibouchi for Barreto--Naehrig curves \cite{FT12} and Wahby--Boneh for
BLS12-381 \cite{WB19}.  Weil and Tate--Lichtenbaum pairings are standard tools
in elliptic-curve cryptography \cite{SilvermanAEC,GPS08}.  Zheng, Lin and Zhao
recently optimized pairing computations on curves with embedding degree two via
biextensions \cite{ZhengLinZhao2026}.  In this paper, however, pairings are used
only to realize group characters in the Fourier analysis of the one-output
distribution.  Other related but distinct goals include
indistinguishable encodings such as Elligator \cite{Elligator} and injective
encodings \cite{FJT13}.

Recent work on one-exponent or one-root maps, including SwiftEC \cite{SwiftEC}
and Koshelev's constructions for ordinary $j=0$ and $j=1728$ curves
\cite{KoshelevJ0,Koshelev1728}, gives useful context.  These maps have a
higher-dimensional source and are designed for near-uniform admissible hashing.
They are not direct competitors to the exact enumeration in this paper.  The
present results concern one classical map $\Fq\to E(\Fq)$ before cofactor
clearing.

\end{document}